\documentclass[12pt]{article}

\title{Symplectic embeddings into four-dimensional concave toric domains}
\author{Keon Choi, Daniel Cristofaro-Gardiner\footnote{Partially supported by NSF grant DMS-0838703.},\\ 
David Frenkel\footnote{Partially supported by SNF grant 200020-144432/1.}, Michael Hutchings\footnote{Partially supported by NSF grant DMS-1105820.},\\ and Vinicius G. B. Ramos\footnote{Partially supported by NSF grant DMS-1105820 and a Harvey Fellowship.}}
\date{}

\usepackage{amssymb}
\usepackage{latexsym}
\usepackage{amsmath}
\usepackage{amsthm}
\usepackage{amscd}
\usepackage[dvips]{graphics}
\usepackage{overpic}

\newcommand{\mc}[1]{{\mathcal #1}}

\numberwithin{equation}{section}

\newtheorem{theorem}{Theorem}[section]
\newtheorem{proposition}[theorem]{Proposition}
\newtheorem{corollary}[theorem]{Corollary}
\newtheorem{lemma}[theorem]{Lemma}
\newtheorem{lemma-definition}[theorem]{Lemma-Definition}

\theoremstyle{definition}
\newtheorem{definition}[theorem]{Definition}
\newtheorem{remark}[theorem]{Remark}

\newtheorem{example}[theorem]{Example}

\newtheorem*{acknowledgments}{Acknowledgments}

\newcommand{\floor}[1]{\left\lfloor #1 \right\rfloor}

\newcommand{\C}{{\mathbb C}}
\newcommand{\Q}{{\mathbb Q}}
\newcommand{\R}{{\mathbb R}}

\newcommand{\Z}{{\mathbb Z}}

\newcommand{\op}{\operatorname}

\newcommand{\Ker}{\op{Ker}}

\newcommand{\bpm}{\begin{pmatrix}}
\newcommand{\epm}{\end{pmatrix}}

\renewcommand{\epsilon}{\varepsilon}

\begin{document}

\setcounter{tocdepth}{2}

\maketitle

\begin{abstract}
ECH capacities give obstructions to symplectically embedding one symplectic four-manifold with boundary into another. We compute the ECH capacities of a large family of symplectic four-manifolds with boundary, called ``concave toric domains''. Examples include the (nondisjoint) union of two ellipsoids in $\R^4$. We use these calculations to find sharp obstructions to certain symplectic embeddings involving concave toric domains. For example: (1) we calculate the Gromov width of every concave toric domain; (2) we show that many inclusions of an ellipsoid into the union of an ellipsoid and a cylinder are ``optimal''; and (3) we find a sharp obstruction to ball packings into certain unions of an ellipsoid and a cylinder.
\end{abstract}

\tableofcontents

\section{Introduction}

\subsection{ECH capacities}

Let $(X,\omega)$ be a symplectic four-manifold, possibly with boundary or corners, noncompact, and/or disconnected. Its ECH capacities are a sequence of real numbers
\begin{equation}
\label{eqn:capacitiesincreasing}
0 = c_0(X,\omega) \le c_1(X,\omega) \le c_2(X,\omega) \le \cdots \le \infty.
\end{equation}
The ECH capacities were introduced in \cite{qech}, see also the exposition in \cite{bn}; we will review the definition in the cases relevant to this paper in \S\ref{sec:capacitiesreview}.

The following are some key properties of ECH capacities:
\begin{description}
\item{(Monotonicity)} If there exists a symplectic embedding $(X,\omega)\to (X',\omega')$, then $c_k(X,\omega)\le c_k(X',\omega')$ for all $k$.
\item{(Conformality)}
If $r>0$ then
\[
c_k(X,r\omega)=rc_k(X,\omega).
\]
\item{(Disjoint union)}
\[
c_k\left(\coprod_{i=1}^n(X_i,\omega_i)\right) = \max_{k_1+\cdots+k_n=k}\sum_{i=1}^n c_{k_i}(X_i,\omega_i).
\]
\item{(Ellipsoid)}
If $a,b>0$, define the ellipsoid
\[
E(a,b)=\left\{(z_1,z_2)\in\C^2\;\bigg|\; \frac{\pi|z_1|^2}{a} + \frac{\pi|z_2|^2}{b} \le 1 \right\}.
\]
Then $c_k(E(a,b))=N(a,b)_k$, where $N(a,b)$ denotes the sequence of all nonnegative integer linear combinations of $a$ and $b$, arranged in nondecreasing order, indexed starting at $k=0$.
\end{description}
Here we are using the standard symplectic form on $\C^2=\R^4$. In particular, define the ball
\[
B(a)=E(a,a).
\]
It then follows from the Ellipsoid property that
\begin{equation}
\label{eqn:ball1}
c_k(B(a)) = ad
\end{equation}
where $d$ is the unique nonnegative integer such that
\begin{equation}
\label{eqn:ball2}
\frac{d^2+d}{2} \le k \le \frac{d^2+3d}{2}.
\end{equation}

It was shown by McDuff \cite{mcd}, see also the survey \cite{pnas}, that there exists a symplectic embedding $\op{int}(E(a,b))\to E(c,d)$ if and only if $N(a,b)_k\le N(c,d)_k$ for all $k$. Thus ECH capacities give a sharp obstruction to symplectically embedding one (open) ellipsoid into another. It follows from work of Frenkel-M\"uller \cite[Prop. 1.4]{fm}, see \cite[Cor.\ 11]{pnas}, that ECH capacities also give a sharp obstruction to symplectically embedding an open ellipsoid into a polydisk
\[
P(a,b) = \left\{(z_1,z_2)\in\C^2\;\big|\;\pi|z_1|^2\le a,\; \pi|z_2|^2\le b\right\}.
\]

On the other hand, ECH capacities do not give sharp obstructions to embedding a polydisk into an ellipsoid. For example, if there is a symplectic embedding $P(1,1)\to E(a,2a)$, then ECH capacities only imply that $a\ge 1$, but the Ekeland-Hofer capacities imply that $a\ge 3/2$, see \cite[Rmk.\ 1.8]{qech}. Another example is that if there is a symplectic embedding from $P(1,2)$ into the ball $B(c)$, then both ECH capacities and Ekeland-Hofer capacities only imply that $c\ge 2$; but in fact it was recently shown by Hind-Lisi \cite{hl} that $c\ge 3$. In particular, the inclusions $P(1,1)\to E(3/2,3)$ and $P(1,2)\to B(3)$ are ``optimal'' in the following sense:

\begin{definition}
A symplectic embedding $\phi:(X,\omega)\to (X',\omega')$ is {\em optimal\/} if there does not exist a symplectic embedding $(X,r\omega)\to (X',\omega')$ for any $r>1$.
\end{definition}

\begin{remark}
\label{remark:optimal}
It follows from the Monotonicity and Conformality properties that if $0<c_k(X,\omega)=c_k(X',\omega')$ for some $k$, and if a symplectic embedding $(X,\omega)\to (X',\omega')$ exists, then it is optimal.
\end{remark}

\subsection{Concave toric domains}
\label{sec:ctd}

We would like to compute more examples of ECH capacities and find more examples of sharp embedding obstructions and optimal symplectic embeddings. An interesting family of symplectic four-manifolds is obtained as follows. If $\Omega$ is a domain in the first quadrant of the plane, define the ``toric domain''
\[
X_\Omega = \left\{z\in\C^2 \;\big|\; \pi(|z_1|^2,|z_2|^2)\in\Omega\right\}.
\]
For example, if $\Omega$ is the triangle with vertices $(0,0)$, $(a,0)$, and $(0,b)$, then $X_\Omega$ is the ellipsoid $E(a,b)$.

The ECH capacities of toric domains $X_\Omega$ when $\Omega$ is convex and does not touch the axes were computed in \cite[Thm.\ 1.11]{qech}, see \cite[Thm.\ 4.14]{bn}. Also, the assumption that $\Omega$ does not touch the axes can be removed in some and conjecturally all cases. In this paper we consider the following new family of toric domains:

\begin{definition}
A {\em concave toric domain\/} is a domain $X_\Omega$ where $\Omega$ is the closed region bounded by the horizontal segment from $(0,0)$ to $(a,0)$, the vertical segment from $(0,0)$ to $(0,b)$, and the graph of a convex function $f:[0,a]\to [0,b]$ with $f(0)=b$ and $f(a)=0$.
The concave toric domain $X_\Omega$ is {\em rational\/} if $f$ is piecewise linear and $f'$ is rational wherever it is defined.
\end{definition}

McDuff showed in \cite[Cor.\ 2.5]{mcd} that the ECH capacities of an ellipsoid $E(a,b)$ with $a/b$ rational are equal to the ECH capacities of a certain ``ball packing'' of the ellipsoid, namely a certain finite disjoint union of balls whose interior symplectically embeds into the ellipsoid filling up all of its volume. These balls are determined by a ``weight expansion'' of the pair $(a,b)$.  In the present work, we generalize this to give a similar formula for the ECH capacities of any rational concave toric domain. In \S\ref{sec:path} we will give a different formula for the ECH capacities of concave toric domains which are not necessarily rational.

\subsection{Weight expansions}
\label{sec:weight}

Let $X_{\Omega}$ be a rational concave toric domain.  The {\em weight expansion\/} of $\Omega$ is a finite unordered list of (possibly repeated) positive real numbers $w(\Omega)=(a_1,\ldots,a_n)$ defined inductively as follows.

If $\Omega$ is the triangle with vertices $(0,0)$, $(a,0)$, and $(0,a)$, then $w(\Omega)=(a)$.

Otherwise, let $a>0$ be the largest real number such that the triangle with vertices $(0,0)$, $(a,0)$, and $(0,a)$ is contained in $\Omega$. Call this triangle $\Omega_1$. The line $x+y=a$ intersects the graph of $f$ in a line segment from $(x_2,a-x_2)$ to $(x_3,a-x_3)$ with $x_2\le x_3$. Let $\Omega_2'$ denote the portion of $\Omega$ above the line $x+y=a$ and to the left of the line $x=x_2$. By first applying the translation $(x,y)\mapsto (x,y-a)$ to $\Omega_2'$ and then multiplying by $\begin{pmatrix}1 & 0 \\ 1 & 1\end{pmatrix}\in SL_2(\Z)$, we obtain a new domain $\Omega_2$ (which we interpret as the empty set if $x_2=0$). Let $\Omega_3'$ denote the portion of $\Omega$ above the line $x+y=a$ and to the right of the line $x=x_3$. By first applying the translation $(x,y)\mapsto (x-a,y)$ and then multiplying by $\begin{pmatrix}1 & 1 \\ 0 & 1\end{pmatrix}\in SL_2(\Z)$, we obtain a new domain $\Omega_3$ (which we interpret as the empty set if $x_3=a$). See Figure~\ref{conc1} for an example of this decomposition. Observe that each $X_{\Omega_i}$ is a rational concave toric domain. We now define
\begin{equation}
\label{eqn:decomposition}
w(\Omega)= w(\Omega_1) \cup w(\Omega_2) \cup w(\Omega_3).
\end{equation}
Here the symbol `$\cup$' indicates ``union with repetitions'', and we interpret $w(\Omega_i)=\emptyset$ if $\Omega_i=\emptyset$. See \S\ref{sec:applications} below for examples of weight expansions.

 \begin{figure}
    \centering
    \begin{overpic}[scale=1.02]{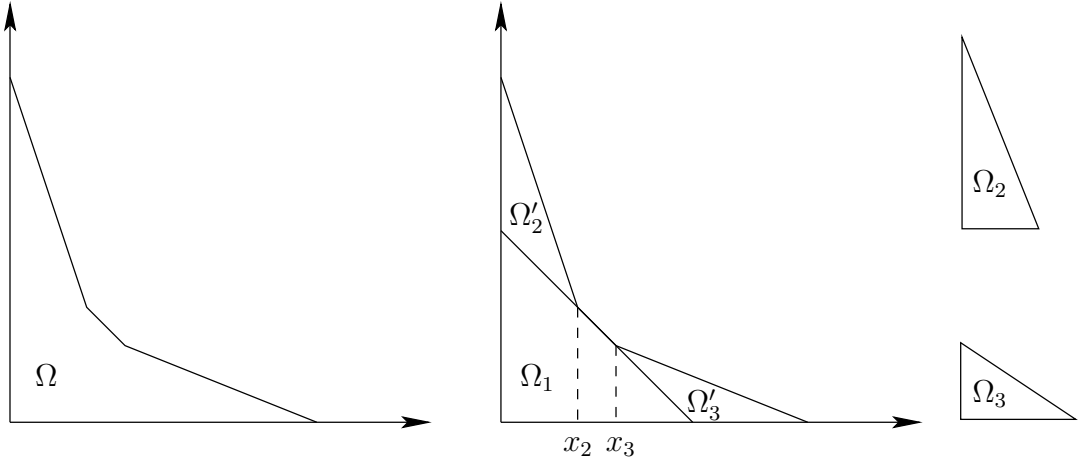}
    \put(3,4){$\Omega$}
    \put(48,4){$\Omega_1$}
    \put(47,19){$\Omega_2'$}
    \put(63.5,1.6){$\Omega_3'$}
    \put(52,-2){$x_2$}
    \put(56,-2){$x_3$}
    \put(90,22){$\Omega_2$}
    \put(90,2.5){$\Omega_3$}
    \end{overpic}
    \vspace{2mm}
    \caption{The inductive step in the decomposition of a concave toric domain}\label{conc1}
    \label{nbhd}
\end{figure}

When $\Omega$ is a rational triangle, the weight expansion is determined by the continued fraction expansion of the slope of the diagonal, and in particular $w(\Omega)$ is finite, see \cite[\S2]{mcd}. If the upper boundary of $\Omega$ has more than one edge, then the upper boundary of each $\Omega_i$ will have fewer edges than that of $\Omega$, so by induction $w(\Omega)$ is still finite.

\begin{theorem}
\label{thm:weight}
The ECH capacities of a rational concave toric domain $X_\Omega$ with weight expansion $(a_1,\ldots,a_n)$ are given by
\[
c_k(X_\Omega) = c_k\left(\coprod_{i=1}^n B(a_i)\right).
\]
\end{theorem}

\begin{remark}
\label{rmk:approximation}
It follows from the Disjoint Union property of ECH capacities, together with the formulas \eqref{eqn:ball1} and \eqref{eqn:ball2} for the ECH capacities of a ball, that
\begin{equation}
\label{eqn:weightcompute}
c_k\left(\coprod_{i=1}^n B(a_i)\right) = \max\left\{\sum_{i=1}^na_id_i \;\bigg|\; \sum_{i=1}^n\frac{d_i^2+d_i}{2} \le k\right\},
\end{equation}
where $d_1,\ldots,d_n$ are nonnegative integers. To compute the maximum on the right hand side of \eqref{eqn:weightcompute}, if we order the weight expansion so that $a_1\ge\cdots\ge a_n$, then we can assume without loss of generality that $d_i=0$ whenever $i>k$.
\end{remark}

\begin{remark}
One can extend Theorem~\ref{thm:weight} to concave toric domains which are not rational; in this case the weight expansion is defined inductively as before, but is now an infinite sequence. To prove this extension of Theorem~\ref{thm:weight}, one can approximate an arbitrary concave toric domain $X_\Omega$ by rational concave toric domains whose weight expansion is the portion of the weight expansion of $X_\Omega$ obtained from the first $n$ steps, and then use the continuity of the ECH capacities in Lemma~\ref{lem:continuity1} below. 
\end{remark}

One inequality in Theorem~\ref{thm:weight} has a quick proof:

\begin{lemma}
\label{lem:weighteasy}
If $X_\Omega$ is a rational concave toric domain with weight expansion $(a_1,\ldots,a_n)$, then
\begin{equation}
\label{eqn:weighteasy}
c_k(X_\Omega) \ge c_k\left(\coprod_{i=1}^n B(a_i)\right).
\end{equation}
\end{lemma}

To prove Lemma~\ref{lem:weighteasy}, we will use the following version of the ``Traynor trick".  Call two domains $\Omega_1$ and $\Omega_2$ in the first quadrant {\em affine equivalent} if one can be obtained from the other by the action of $SL_2(\Z)$ and translation.  Let $\bigtriangleup(a)$ denote the open triangle with vertices $(0,0)$, $(a,0)$, and $(0,a)$.

\begin{lemma}
\label{lem:traynortrick}
If $T$ is an open triangle in the first quadrant which is affine equivalent to $\bigtriangleup(a)$, then there exists a symplectic embedding $\op{int}(B(a))\to X_T$.
\end{lemma}

\begin{proof}
It follows from \cite[Prop. 5.2]{traynor} that there exists a symplectic embedding
\[
\op{int}(B(a)) \to X_{\bigtriangleup(a)}.
\]  
On the other hand, 
if $\Omega_1$ and $\Omega_2$ are affine equivalent and do not contain any points on the axes, then $X_{\Omega_1}$ is symplectomorphic to $X_{\Omega_2}$.
Thus $X_{\bigtriangleup(a)}$ is symplectomorphic to $X_{T}$ and  we are done.
\end{proof}
 
\begin{proof}[Proof of Lemma~\ref{lem:weighteasy}.]
It follows from the definition of the weight expansion that $\Omega$ has a decomposition into open triangles $T_1,\ldots,T_n$ such that $T_i$ is affine equivalent to $\bigtriangleup(a_i)$ for each $i$.  By Lemma~\ref{lem:traynortrick}, for each $i$ there is a symplectic embedding $\op{int}(B(a_i)) \to X_{T_i}$. Hence there is a symplectic embedding
\[
\coprod_{i=1}^n\op{int}(B(a_i)) \to X_\Omega.
\]
It then follows from the Monotonicity property of ECH capacities that \eqref{eqn:weighteasy} holds.
\end{proof}

\subsection{Examples and first applications}
\label{sec:applications}

We now give some examples of how Theorem~\ref{thm:weight} can be used to prove that certain symplectic embeddings are optimal.

The following lemma will be helpful. If $\ell$ is a nonnegative integer, define $w_\ell(\Omega)\subset w(\Omega)$ to be the list of positive real numbers obtained from the first $\ell$ steps in the inductive construction of the weight expansion. That is, $w_0(\Omega) = \emptyset$ and 
\[
w_\ell(\Omega) = w(\Omega_1) \cup w_{\ell-1}(\Omega_2) \cup w_{\ell-1}(\Omega_3)
\]
for $\ell > 0$.


\begin{lemma}
\label{lem:approximation}
If $w_{\ell}(\Omega)=(a_1,\ldots,a_m)$, then for any $k \le \ell$,  
\[
c_k(X_\Omega) = c_k\left(\coprod_{i=1}^m B(a_i)\right).
\]
\end{lemma}

\begin{proof}
Let $(a_1,\ldots,a_n)$ be the weight expansion for $\Omega$.  By Theorem~\ref{thm:weight}, it is enough to prove that
\begin{equation}
\label{eqn:weightab}
c_k\left(\coprod_{i=1}^n B(a_i)\right) = c_k\left(\coprod_{i=1}^m B(a_i)\right).
\end{equation}
By Remark~\ref{rmk:approximation}, the left hand side of \eqref{eqn:weightab} is determined by the $k$ largest numbers in $w(\Omega)$, and the right hand side of \eqref{eqn:weightab} is determined by the $k$ largest numbers in $w_\ell(\Omega)$. It follows from the definition of the weight expansion and induction that the $k$ largest numbers in $w(\Omega)$ are a subset of $w_k(\Omega)$; and the latter is a subset of $w_\ell(\Omega)$ since $k\le \ell$. Thus the two sides of \eqref{eqn:weightab} are equal.
\end{proof}

We now have the following corollary of Theorem~\ref{thm:weight}.

\begin{corollary}
\label{cor:gromovwidth}
If $X_\Omega$ is a rational concave toric domain, let $a$ be the largest real number such that $B(a)\subset X_\Omega$. Then the inclusion $B(a)\subset X_\Omega$ is optimal, so the Gromov width of $X_\Omega$ equals $a$.
\end{corollary}

\begin{proof}
Note that $a$ is just the largest real number such that 
$\bigtriangleup(a)\subset \Omega$. It follows from Lemma~\ref{lem:approximation} with $\ell=1$ that $c_1(X_\Omega)=a$.  Since $c_1(B(a))=a$, we are done by Remark~\ref{remark:optimal}.
\end{proof}

Here is a simple example of obstructions to symplectic embeddings in which $X_\Omega$ is the domain rather than the target:

\begin{example}
Let $a\in(0,1)$, and let $\Omega$ be the quadrilateral with vertices $(0,0)$, $(1,0)$, $(a,1-a)$ and $(0,1+a)$. Then the inclusion $X_\Omega\subset B(1+a)$ is optimal.
\end{example}

\begin{proof}
The weight expansion is $w(\Omega)=(1,a)$. It then follows from equation \eqref{eqn:weightcompute} that $c_2(X_\Omega)=1+a$. Since $c_2(B(1+a))=1+a$, the claim follows from Remark~\ref{remark:optimal}.
\end{proof}

Another interesting example is the (nondisjoint) union of a ball and a cylinder. Given $0<a<b$, define $Z(a,b)$ to be the union of the ball $B(b)$ with the cylinder
\[
Z(a)=P(\infty,a).
\]
That is, $Z(a,b)=X_\Omega$ where $\Omega$ is bounded by the axes, the line segment from $(0,b)$ to $(b-a,a)$, and the horizontal ray extending to the right from $(b-a,a)$.

\begin{proposition}
\label{prop:bc}
The ECH capacities of the union of a ball and a cylinder are given by
\begin{equation}
\label{eqn:bc}
c_k(Z(a,b))=\max\left\{db+a\left(k-\frac{d(d+1)}{2}\right)\;\bigg|\;  d(d+1)\le 2k\right\}
\end{equation}
where $d$ is a nonnegative integer.
\end{proposition}

\begin{proof}
Recall from \cite[\S4.2]{qech} that for any symplectic four-manifold $(X,\omega)$, we have
\begin{equation}
\label{eqn:Xminus}
c_k(X,\omega) = \sup\left\{c_k\left(X_-,\omega|_{X_-}\right)\right\}
\end{equation}
where the supremum is over certain compact subsets $X_-\subset \op{int}(X)$ (namely those for which $(X_-,\omega|_{X_-})$ is a four-dimensional ``Liouville domain'' in the sense of \cite[\S1]{qech}).
 It follows immediately that ECH capacities have the following ``exhaustion property'': if $\{X_i\}_{i\ge 1}$ is a sequence of open subsets of $X$ with $X_i\subset X_{i+1}$ and $\bigcup_{i=1}^\infty X_i=\op{int}(X)$, then
\begin{equation}
\label{eqn:exhaustion}
c_k(X,\omega) = \lim_{i\to \infty}c_k\left(X_i,\omega|_{X_i}\right).
\end{equation}

To apply this in the present situation, given a positive integer $i$, let $\Omega_i$ be the quadrilateral with vertices $(0,0)$, $(0,b)$, $(b-a,a)$, and $(b+ia,0)$.  Then the interiors of the domains $X_{\Omega_i}$ exhaust the interior of $Z(a,b)$. Also, $X_{\Omega_i}$ has the same ECH capacities as its interior; this follows for example from \eqref{eqn:Xminus}. It then follows from the exhaustion property \eqref{eqn:exhaustion} that
\begin{equation}
\label{eqn:exhaustion2}
c_k(Z(a,b)) = \lim_{i\to \infty}c_k(X_{\Omega_i}).
\end{equation}

Assume that $i\ge k$. We now compute $c_k(X_{\Omega_i})$ using Theorem~\ref{thm:weight}. The weight expansion of $\Omega_i$ is
\begin{equation}
\label{eqn:womegai}
w(\Omega_i) = (b,\underbrace{a,\ldots,a}_{\mbox{$i$ times}}).
\end{equation}
Since $i\ge k$, to compute the maximum in \eqref{eqn:weightcompute}, we can assume that each $a$ weight in \eqref{eqn:womegai} is multiplied by $0$ or $1$, and the $b$ weight in \eqref{eqn:womegai} is multiplied by $(d^2+d)/2$ for some nonnegative integer $d$. It then follows that $c_k(X_{\Omega_i})$ equals the right hand side of \eqref{eqn:bc}.
It now follows from \eqref{eqn:exhaustion2} that \eqref{eqn:bc} holds.
\end{proof}

It is interesting to ask when the ellipsoid $E(a,b)$ symplectically embeds into $Z(c,d)$. By scaling, it is equivalent to ask, given $a,b\ge 1$, for which $\lambda>0$ there exists a symplectic embedding $E(a,1)\to Z(\lambda,\lambda b)$. Of course this trivially holds if $\lambda$ is sufficiently large that $E(a,1)$ is a subset of $Z(\lambda,\lambda b)$. In some cases this sufficient condition is also necessary:

\begin{corollary}
\label{cor:ballcylinder}
Suppose that (i) $a \in \lbrace 1,2\rbrace$ and $b\ge 1$, or (ii) $a$ is a positive integer and $1\le b\le 2$. Then there exists a symplectic embedding $E(a,1)\to Z(\lambda,\lambda b)$ if and only if $E(a,1)\subset Z(\lambda,\lambda b)$.
\end{corollary}

\begin{proof}
We first compute that $E(a,1)\subset Z(\lambda,\lambda b)$ if and only if
\begin{equation}
\label{eqn:lambda1}
\lambda \ge \frac{a}{a+b-1}.
\end{equation}
Assuming (i) or (ii), we need to show that if there exists a symplectic embedding $E(a,1)\to Z(\lambda,\lambda b)$, then the inequality \eqref{eqn:lambda1} holds. By the Monotonicity and Conformality properties of ECH capacities, it will suffice to show that
\begin{align}
\label{eqn:cae}
c_a(E(a,1)) &= a,\\
\label{eqn:cab}
c_a(Z(1,b))&= a+b-1.
\end{align}
Now \eqref{eqn:cae} holds for any positive integer $a$ by the Ellipsoid property.  And in both cases (i) and (ii), equation \eqref{eqn:cab} follows from Proposition~\ref{prop:bc}, because the maximum in \eqref{eqn:bc} is realized by $d=1$.
\end{proof}

\begin{remark}
There are many cases in which an ellipsoid $E(a,1)$ symplectically embeds into $Z(\lambda,\lambda b)$ although $E(a,1)$ is not a subset of $Z(\lambda,\lambda b)$. 
 For example, an ellipsoid $E(a,1)$ may embed into a ball $B(c)$ of slightly greater volume, and this is always possible when $a\ge (17/6)^2$, see \cite{MS}; if we set $c=\lambda b$, then the ellipsoid is not a subset of $Z(\lambda,\lambda b)$ if we choose $b$ sufficiently large.  Moreover, the ``symplectic folding" method from \cite{s} can be used to construct examples of symplectic embeddings $E(a,1)\to Z(\lambda,\lambda b)$ where $E(a,1)\not\subset Z(\lambda,\lambda b)$ and also $\op{vol}(E(a,1))>\op{vol}(B(\lambda b))$, so that $E(a,1)$ does not symplectically embed into the ball $B(\lambda b)$ alone. 
\end{remark}

Corollary~\ref{cor:ballcylinder} also has a generalization to symplectic embeddings of an ellipsoid into the union of an ellipsoid and a cylinder, see \S\ref{sec:oee}.

\subsection{Application to ball packings}
\label{sec:packingintro}

As a more involved application, we obtain a sharp obstruction to ball packings of the union of certain unions of a cylinder and an ellipsoid. Given positive real numbers $a, b$ and $c$ with $c > a$, define 
\[
Z(a,b,c)=Z(a) \cup E(b,c).
\]   

\begin{theorem}
\label{thm:packing}
Let $b,c$ and $w_1\ge w_2\ge \cdots \ge w_n>0$ be positive real numbers.  Assume that $c > 1$ and $b\le \frac{c}{c-1}$.  Then there exists a symplectic embedding   
\[
\coprod_{i=1}^{n} \op{int}(B\left(w_{i}\right)) \to Z(\lambda,\lambda b,\lambda c)
\]
if and only if 
\[
\lambda \ge \max\{w_1/c,\lambda_1,\ldots,\lambda_n\},
\]
where we define
\begin{equation}
\label{eqn:lambdak}
\lambda_k = \frac{\sum_{i=1}^{k}w_{i}}{k+\frac{b(c-1)}{c}}.
\end{equation}     
\end{theorem}

For example, Theorem~\ref{thm:packing} gives a sharp obstruction to embedding a disjoint union of balls into the union of a ball and a cylinder, $Z(a,b)=Z(a,b,b)$, as long as $b \le 2a$. 
 
The outline of the proof of Theorem~\ref{thm:packing} is as follows. In \S\ref{sec:cbp}, we will give a symplectic embedding construction to prove:

\begin{proposition}
\label{prop:packing}
Let $b,c$ and $w_1\ge w_2\ge \cdots \ge w_n>0$ be positive real numbers.  Assume that $c > 1$.  Define $\lambda_k$ by \eqref{eqn:lambdak}. If 
\[
\lambda \ge \max\{w_1/c,\lambda_1,\ldots,\lambda_n\},
\]
then there exists a symplectic embedding
\begin{equation}
\label{eqn:shearing}   
\coprod_{i=1}^{n} \op{int}(B\left(w_{i}\right)) \to Z(\lambda, \lambda b,\lambda c).
\end{equation}
\end{proposition}

This implies the sufficient condition for ball packings in Theorem~\ref{thm:packing}. We will then use ECH capacities to prove the necessary condition for ball packings in Theorem~\ref{thm:packing}.

\begin{remark}
Unlike Theorem~\ref{thm:packing}, Proposition~\ref{prop:packing} still holds when $b> \frac{c}{c-1}$, but in this case we generally do not know whether better symplectic embeddings are possible. For example, Proposition~\ref{prop:packing} implies that one can symplectically embed three equal balls $\op{int}(B(a))$ into $Z(1,3)$ whenever $a\le 5/3$. However ECH capacities only tell us that if such an embedding exists then $a\le 2$. 
\end{remark}

\subsection{ECH capacities and lattice points}
\label{sec:path}

We now give a different formula for the ECH capacities of a concave toric domain, which is not assumed to be rational. This formula requires the following definitions.

\begin{definition}
A {\em concave integral path\/} is a polygonal path in the plane, whose vertices are at lattice points, and which is the graph of a convex piecewise linear function $F:[0,B]\to[0,A]$ for some nonnegative integers $A,B$.
\end{definition}

\begin{definition}
If $\Lambda$ is a concave integral path, define $\mc{L}(\Lambda)$ to be the number of lattice points in the region bounded by $\Lambda$, the line segment from $(0,0)$ to $(0,B)$, and the line segment from $(0,0)$ to $(A,0)$, not including lattice points on $\Lambda$ itself.
\end{definition}

\begin{definition}
If $X_\Omega$ is the concave toric domain determined by $f:[0,b]\to[a,0]$, and if $\Lambda$ is a concave integral path, define the {\em $\Omega$-length\/} of $\Lambda$, denoted by $\ell_\Omega(\Lambda)$, as follows. For each edge $e$ of $\Lambda$, let $v_e$ denote the vector determined by $e$, namely the difference between the right and left endpoints. Let $p_e$ be a point on the graph of $f$ such that the graph of $f$ is contained in the closed half-plane above the line through $p_e$ parallel to $e$. Then
\begin{equation}
\label{eqn:ellomega}
\ell_\Omega(\Lambda) = \sum_{e\in\op{Edges}(\Lambda)} v_e \times p_e.
\end{equation}
Here  $\times$ denotes the cross product. Note that $p_e$ fails to be unique only when the graph of $f$ contains an edge parallel to $e$, in which case $v_e \times p_e$ does not depend on the choice of $p_e$.
\end{definition}

\begin{theorem}
\label{thm:path}
If $X_\Omega$ is any concave toric domain, then its ECH capacities are given by
\begin{equation}
\label{eqn:path}
c_k(X_\Omega) = \max\{\ell_\Omega(\Lambda) \mid \mc{L}(\Lambda) = k\}.
\end{equation}
Here the maximum is over concave integral paths $\Lambda$.
\end{theorem}

\begin{remark}
It is interesting to compare Theorem~\ref{thm:path} with the formula for the ECH capacities of {\em convex\/} toric domains in \cite[Thm.\ 4.14]{bn}, in which one {\em minimizes\/} a length function over convex paths enclosing a certain number of lattice points.
\end{remark}

\begin{example}
\label{ex:checkell}
Let us check that Theorem~\ref{thm:path} correctly recovers $c_k(X_\Omega)$ when $\Omega$ is the triangle with vertices $(0,0)$, $(a,0)$ and $(0,b)$, so that $X_\Omega = E(a,b)$.

An equivalent statement of the Ellipsoid property is that $c_k(E(a,b))=L_k$ where $L_k$ is the smallest nonnegative real number such that triangle bounded by the axes and the line $bx+ay=L_k$ encloses at least $k+1$ lattice points. Call this triangle $T_k$, and call its upper edge $D_k$.

To see that $L_k$ agrees with the right hand side of \eqref{eqn:path}, suppose first that $a/b$ is irrational.  There is then a unique lattice point $(x_k,y_k)$ on $D_k$. We need to show that
\begin{equation}
\label{eqn:checkell}
\max\{\ell_\Omega(\Lambda) \mid \mc{L}(\Lambda) = k\} = bx_k+ay_k.
\end{equation}
If $\Lambda$ is a concave integral path, there is a unique vertex $(x,y)\in\Lambda$ such that $\Lambda$ is contained in the closed half-plane above the line through $(x,y)$ with slope $-b/a$. Then $p_e=(0,b)$ for all edges to the left of $(x,y)$, and $p_e=(a,0)$ for all edges to the right of $(x,y)$. Therefore
\[
\ell_\Omega(\Lambda) = bx + ay.
\]
If $\mc{L}(\Lambda)\le k$, then we must have $bx+ay\le bx_k+ay_k$, since otherwise every lattice point in $T_k$ would be counted by $\mc{L}(\Lambda)$. Thus the left hand side of \eqref{eqn:checkell} is less than or equal to the right hand side. To prove the reverse inequality, observe that if $\Lambda$ is the minimal concave integral path which contains the point $(x_k,y_k)$ and is contained in the closed half-plane above the line $D_k$, then $(x,y)=(x_k,y_k)$ and $\mc{L}(\Lambda)=k$.

Suppose now that $a/b$ is rational. Then $D_k$ may contain more than one lattice point. If $\Lambda$ is a concave integral path, then there is a unique pair of (possibly equal) vertices $(x,y),(x',y')\in\Lambda$ with $x\le x'$ such that line segment from $(x,y)$ to $(x',y')$ is contained in
$\Lambda$, and the rest of $\Lambda$ is strictly above the line through $(x,y)$ with slope $-b/a$. Now if $p$ is any point on the upper edge of $\Omega$, then we have
\[
\ell_\Omega(\Lambda) = bx + (x'-x,y'-y)\times p + ay'.
\]
We can choose $p=(a,0)$ for convenience, and this gives
\[
\ell_\Omega(\Lambda) = bx + ay.
\]
The rest of the argument in this case is similar to the previous case.

One can also deduce the case when $a/b$ is rational from the case when $a/b$ is irrational by a continuity argument using Lemma~\ref{lem:continuity2} below.
\end{example}

\subsection{The rest of the paper}

Theorems~\ref{thm:weight} and \ref{thm:path}, which compute the ECH capacities of concave toric domains, are proved in \S\ref{sec:lb} and \S\ref{sec:ub}. The generalization of Corollary~\ref{cor:ballcylinder} to symplectic embeddings of an ellipsoid into the union of an ellipsoid and a cylinder is given in \S\ref{sec:oee}.  Proposition~\ref{prop:packing} and Theorem~\ref{thm:packing} on ball packings of the union of an ellipsoid and a cylinder are proved in \S\ref{sec:cbp} and \S\ref{sec:obp}.  

\begin{acknowledgments}
It is a pleasure to thank Daniel Irvine and Felix Schlenk for many helpful discussions.      
\end{acknowledgments}

\section{The lower bound on the capacities}
\label{sec:lb}

In this section we use combinatorial arguments to prove half of Theorem~\ref{thm:path}, namely:

\begin{lemma}
\label{lem:pathlb}
If $X_\Omega$ is any concave toric domain, then
\begin{equation}
\label{eqn:pathlb}
c_k(X_\Omega) \ge \max\{\ell_\Omega(\Lambda) \mid \mc{L}(\Lambda) = k\}.
\end{equation}
Here the maximum is over concave integral paths $\Lambda$.
\end{lemma}

\subsection{The lower bound in the rational case}

The following lemma, together with Lemma~\ref{lem:weighteasy},
implies Lemma~\ref{lem:pathlb} in the rational case.

\begin{lemma}
\label{lem:rationallb}
Let $X_\Omega$ be a rational concave toric domain with weight expansion $(a_1,\ldots,a_n)$. Then
\begin{equation}
\label{eqn:rationallb}
c_k\left(\coprod_{i=1}^n B(a_i)\right) \ge \max\{\ell_\Omega(\Lambda) \mid \mc{L}(\Lambda)=k\}.
\end{equation}
\end{lemma}

\begin{proof} The proof has four steps.

{\em Step 1: Setup.\/}
We use induction on $n$. If $n=1$, then
$X_\Omega$ is a ball and we know from Example~\ref{ex:checkell} that both sides of \eqref{eqn:rationallb} are equal. If $n>1$, let $\Omega_1$, $\Omega_2$, and $\Omega_3$ be as in the definition of the weight expansion in \S\ref{sec:weight}. By induction, we can assume that the lemma is true for $\Omega_1$, $\Omega_2$, and $\Omega_3$.

Let $\Lambda$ be a concave integral path with $\mc{L}(\Lambda) = k$. We need to show that
\begin{equation}
\label{eqn:lbneed}
c_k\left(\coprod_{i=1}^n B(a_i)\right) \ge \ell_\Omega(\Lambda).
\end{equation}
To prove this, let $W_i$ denote the disjoint union of the balls given by the weight expansion of $\Omega_i$ for $i=1,2,3$. By the definition of the weight expansion we have
\begin{equation}
\label{eqn:bdwe}
\coprod_{i=1}^n B(a_i) = \coprod_{i=1}^3 W_i.
\end{equation}
In Step 2 we will define concave integral paths $\Lambda_i$ for $i=1,2,3$, and we write $k_i=\mc{L}(\Lambda_i)$. By \eqref{eqn:bdwe} and the Disjoint Union property of ECH capacities, we know that
\[
c_{k_1+k_2+k_3}\left(\coprod_{i=1}^n B(a_i)\right) \ge \sum_{i=1}^3 c_{k_i}(W_i).
\]
By the inductive hypothesis we know that
\[
c_{k_i}(W_i) \ge \ell_{\Omega_i}(\Lambda_i).
\]
In Steps 3 and 4 we will further show that
\begin{equation}
\label{eqn:lb1}
k_1+k_2+k_3=k
\end{equation}
and
\begin{equation}
\label{eqn:lb2}
\sum_{i=1}^3\ell_{\Omega_i}(\Lambda_i)=\ell_\Omega(\Lambda).
\end{equation}
The above four equations and inequalities then imply \eqref{eqn:lbneed}.

{\em Step 2: Definition of $\Lambda_i$.} The paths $\Lambda_i$ are obtained from $\Lambda$ in the same way that the domains $\Omega_i$ are obtained from $\Omega$. We now make this explicit in order to fix notation. Let $\Lambda_1$ be the maximal line segment with slope $-1$ from the $y$ axis to the $x$ axis such that $\Lambda$ is contained in the closed half-space above the line extending $\Lambda_1$. Let $(0,A)$ and $(A,0)$ denote the endpoints of $\Lambda_1$. Let $\Lambda_2'$ denote the portion of $\Lambda$ to the left of $\Lambda_1\cap\Lambda$, and let $\Lambda_3'$ denote the portion of $\Lambda$ to the right of $\Lambda_1\cap\Lambda$. Let $T_2:\R^2\to\R^2$ be the map obtained by first translating down by $A$ and then multiplying by $\begin{pmatrix} 1 & 0 \\ 1 & 1\end{pmatrix} \in SL_2\Z$. Then $\Lambda_2 = T_2(\Lambda_2')$. Similarly, $\Lambda_3=T_3(\Lambda_3')$, where $T_3:\R^2\to\R^2$ is the map obtained by first translating to the left by $A$ and then multiplying by $\begin{pmatrix}1&1\\0&1\end{pmatrix} \in SL_2\Z$.

{\em Step 3: Proof of equation \eqref{eqn:lb1}.\/} Since $T_2$ preserves the lattice, $\mc{L}(\Lambda_2)$ is the number of lattice points counted by $\mc{L}(\Lambda)$ that are on or above $\Lambda_1$ and below $\Lambda_2'$. Likewise, $\mc{L}(\Lambda_3)$ is the number of lattice points counted by $\mc{L}(\Lambda)$ that are on or above $\Lambda_1$ and below $\Lambda_3'$. The remaining lattice points counted by $\mc{L}(\Lambda)$ are those that are below $\Lambda_1$, which are exactly the lattice points counted by $\mc{L}(\Lambda_1)$.

{\em Step 4: Proof of equation \eqref{eqn:lb2}.\/}
By construction, there is an injection
\[
\phi:
\op{Edges}(\Lambda) \to \coprod_{i=1}^3 \op{Edges}(\Lambda_i).
\]
The complement of the image of this injection consists of those edges of $\Lambda_1$ that are to the left or to the right of $\Lambda_1\cap\Lambda$. Denote these two sets of edges by $\op{Left}(\Lambda_1)$ and $\op{Right}(\Lambda_1)$ respectively. We tautologically have
\begin{equation}
\label{eqn:center}
\sum_{e\in\phi^{-1}(\op{Edges}(\Lambda_1))}v_e\times p_e = \left(\sum_{\hat{e}\in\op{Edges}(\Lambda_1)} - \sum_{\hat{e}\in\op{Left}(\Lambda_1)} - \sum_{\hat{e}\in\op{Right}(\Lambda_1)}\right) v_{\hat{e}}\times p_{\hat{e}}.
\end{equation}
Here if $\hat{e}$ is an edge of $\Lambda_i$, then $p_{\hat{e}}$ denotes the (not necessarily unique) point on the upper edge of $\Omega_i$ that appears in the formula \eqref{eqn:ellomega} for $\ell_{\Omega_i}(\Lambda_i)$.
To prove equation \eqref{eqn:lb2}, it is enough to show in addition to \eqref{eqn:center} that
\begin{equation}
\label{eqn:left}
\sum_{e\in\phi^{-1}(\op{Edges}(\Lambda_2))} v_e\times p_e = \sum_{\hat{e}\in\op{Edges}(\Lambda_2)}v_{\hat{e}}\times p_{\hat{e}} + \sum_{\hat{e}\in\op{Left}(\Lambda_1)}v_{\hat{e}}\times p_{\hat{e}}
\end{equation}
and
\begin{equation}
\label{eqn:right}
\sum_{e\in\phi^{-1}(\op{Edges}(\Lambda_3))} v_e\times p_e = \sum_{\hat{e}\in\op{Edges}(\Lambda_3)}v_{\hat{e}}\times p_{\hat{e}} + \sum_{\hat{e}\in\op{Right}(\Lambda_1)}v_{\hat{e}}\times p_{\hat{e}}.
\end{equation}

We will just prove equation \eqref{eqn:left}, as the proof of \eqref{eqn:right} is analogous. Let $e\in\phi^{-1}(\op{Edges}(\Lambda_2))$ and let $\hat{e}=\phi(e)$. We then have
\[
v_{\hat{e}} = \begin{pmatrix}1&0\\1&1\end{pmatrix} v_e
\]
and
\[
p_{\hat{e}} = \begin{pmatrix}1&0\\1&1\end{pmatrix}(p_e - (0,a))
\]
where $a$ is as in the definition of the weight expansion of $\Omega$ in \S\ref{sec:weight}. Consequently,
\[
v_e\times p_e = v_{\hat{e}}\times p_{\hat{e}} + v_e\times (0,a).
\]
Summing over all $e\in\phi^{-1}(\op{Edges}(\Lambda_2))$ gives
\begin{equation}
\label{eqn:left'}
\sum_{e\in\phi^{-1}(\op{Edges}(\Lambda_2))} v_e\times p_e = \sum_{\hat{e}\in\op{Edges}(\Lambda_2)}v_{\hat{e}}\times p_{\hat{e}} + \sum_{e\in\phi^{-1}(\op{Edges}(\Lambda_2))}v_{e}\times (0,a).
\end{equation}
But the rightmost sum in \eqref{eqn:left'} agrees with the rightmost sum in \eqref{eqn:left}, because for $\hat{e}\in\Lambda_1$ one can take $p_{\hat{e}}=(0,a)$, and the total horizontal displacement of the edges in $\phi^{-1}(\op{Edges}(\Lambda_2))$ is the same as the total horizontal displacement of the edges in $\op{Left}(\Lambda_1)$.
\end{proof}

\subsection{Continuity}

Having proved the lower bound \eqref{eqn:pathlb} for rational concave toric domains, we now use a continuity argument to extend this bound to arbitrary concave toric domains.

Recall that the Hausdorff metric on compact subsets of $\R^2$ is defined by
\[
d(\Omega_1,\Omega_2) = \max_{p_1\in\Omega_1}\min_{p_2\in\Omega_2}d(p_1,p_2) + \max_{p_2\in\Omega_2}\min_{p_1\in\Omega_1}d(p_2,p_1).
\]

\begin{lemma}
\label{lem:continuity1}
If $k$ is fixed, then $c_k(X_\Omega)$ is a continuous function of $\Omega$ with respect to the Hausdorff metric.
\end{lemma}

\begin{proof}
Fix $\Omega$, and given $r>0$, consider the scaling $r\Omega=\{(rx,ry)\mid (x,y)\in\Omega\}$. Observe that $X_{r\Omega}$ is symplectomorphic to $X_\Omega$ with the symplectic form multiplied by $r$. It then follows from the Conformality property of ECH capacities that $c_k(X_{r\Omega})=rc_k(X_\Omega)$. If $\{\Omega_i\}_{i\ge 1}$ is a sequence converging to $\Omega$ in the Hausdorff metric, then there is a sequence of positive real numbers $\{r_i\}_{i\ge 1}$ converging to $1$ such that
\[
r_i^{-1}\Omega\subset\Omega_i\subset r_i\Omega.
\]
By the Monotonicity property of ECH capacities, we have
\[
r_i^{-1}c_k(X_\Omega) \le c_k(X_{\Omega_i})\le r_i c_k(X_\Omega).
\]
It follows that $\lim_{i\to\infty}c_k(X_{\Omega_i})=c_k(X_\Omega)$.
\end{proof}

\begin{lemma}
\label{lem:continuity2}
If $k$ is fixed, then $\max\{\ell_\Omega(\Lambda)\mid \mc{L}(\Lambda)=k\}$ is a continuous function of $\Omega$ with respect to the Hausdorff metric.
\end{lemma}

\begin{proof}
For $k$ fixed, there are only finitely many concave integral paths $\Lambda$ with $\mc{L}(\Lambda) = k$. Consequently, it is enough to show that if $\Lambda$ is a fixed concave integral path, then $\ell_\Omega(\Lambda)$ is a continuous function of $\Omega$. By \eqref{eqn:ellomega}, it is now enough to show that if $e$ is an edge of $\Lambda$, then $v_e\times p_e(\Omega)$ is a continuous function of $\Omega$. In fact there is a constant $c>0$ depending only on $v_e$ such that
\[
|v_e\times p_e(\Omega) - v_e\times p_e(\Omega')|\le c d(\Omega,\Omega').
\]
To see this, suppose that $v_e\times p_e(\Omega) < v_e\times p_e(\Omega')$. Write $p_e(\Omega)=(x_0,y_0)$. Every point $(x,y)\in\Omega$ must have $x\le x_0$ or $y\le y_0$. The portion of the upper boundary of $\Omega'$ with $x\ge x_0$ and $y\ge y_0$ is a path from the line $x=x_0$ to the line $y=y_0$. Let $p'\in\Omega'$ denote the intersection of this path with the line of slope $1$ through the point $(x_0,y_0)$. The above path must stay above the triangle bounded by the line $x=x_0$, the line $y=y_0$, and the line through $p_e(\Omega')$ parallel to $v_e$. It follows that there is a constant $c'$ depending only on $v_e$ such that
\[
\min_{p\in\Omega}d(p',p) \ge c' v_e\times(p_e(\Omega')-p_e(\Omega)).
\]
\end{proof}

\begin{proof}[Proof of Lemma~\ref{lem:pathlb}.]
By Lemmas~\ref{lem:weighteasy} and \ref{lem:rationallb}, this holds for rational concave toric domains. The general case now follows from Lemmas~\ref{lem:continuity1} and \ref{lem:continuity2}, since if $X_\Omega$ is an arbitrary concave toric domain, then $\Omega$ can be approximated in the Hausdorff metric by $\Omega'$ such that $X_{\Omega'}$ is a rational concave toric domain. 
\end{proof}

\section{The upper bound on the capacities}
\label{sec:ub}

To complete the proofs of Theorems~\ref{thm:weight} and \ref{thm:path}, we now prove:

\begin{lemma}
\label{lem:pathub}
If $X_\Omega$ is any concave toric domain, then
\begin{equation}
\label{eqn:pathub}
c_k(X_\Omega) \le \max\{\ell_\Omega(\Lambda) \mid \mc{L}(\Lambda) = k\}.
\end{equation}
Here the maximum is over concave integral paths $\Lambda$.
\end{lemma}

Note that Theorem~\ref{thm:weight} follows by combining Lemmas~\ref{lem:weighteasy}, \ref{lem:rationallb}, and \ref{lem:pathub}, while Theorem~\ref{thm:path} follows by combining Lemmas~\ref{lem:pathlb} and \ref{lem:pathub}.

\subsection{ECH capacities of star-shaped domains}
\label{sec:capacitiesreview}

The proof of Lemma~\ref{lem:pathub} requires some knowledge of the definition of ECH capacities, which we now briefly review; for full details see \cite{qech} or \cite{bn}. We will only explain the definition for the special case of smooth star-shaped domains in $\R^4$, since that is what we need here.

Let $Y$ be a three-manifold diffeomorphic to $S^3$, and let $\lambda$ be a nondegenerate contact form on $Y$ such that $\Ker(\lambda)$ is the tight contact structure. The {\em embedded contact homology\/} $ECH_*(Y,\lambda)$ is the homology of a chain complex $ECC_*(Y,\lambda,J)$ over $\Z/2$ defined as follows. (ECH can also be defined with integer coefficients, see \cite[\S9]{obg2}, but that is not needed for the definition of ECH capacities.) A generator of the chain complex is a finite set of pairs $\alpha=\{(\alpha_i,m_i)\}$ where the $\alpha_i$ are distinct embedded Reeb orbits, the $m_i$ are positive integers, and $m_i=1$ whenever $\alpha_i$ is hyperbolic. The chain complex in this case has an absolute $\Z$-grading which is reviewed in \S\ref{sec:approx} below; the grading of a generator $\alpha$ is denoted by $I(\alpha)\in\Z$. The chain complex differential counts certain $J$-holomorphic curves in $\R\times Y$ for an appropriate almost complex structure $J$; the precise definition of the differential is not needed here. Taubes \cite{taubes} proved that the embedded contact homology of a contact three-manifold is isomorphic to a version of its Seiberg-Witten Floer cohomology as defined by Kronheimer-Mrowka \cite{km}. For the present case of $S^3$ with its tight contact structure, this implies that
\[
ECH_*(Y,\lambda) = \left\{\begin{array}{cl} \Z/2, & *=0,2,4,\ldots,\\ 0, & \mbox{otherwise}. \end{array}\right.
\]
We denote the nonzero element of $ECH_{2k}(Y,\lambda)$ by $\zeta_k$.

The {\em symplectic action\/} of a chain complex generator $\alpha=\{(\alpha_i,m_i)\}$ is defined by
\[
\mc{A}(\alpha) = \sum_im_i\int_{\alpha_i}\lambda.
\]
We define $c_k(Y,\lambda)$ to be the smallest $L\in\R$ such that $\zeta_k$ has a representative in $ECC_*(Y,\lambda,J)$ which is a sum of chain complex generators each of which has symplectic action less than or equal to $L$. It follows from \cite[Thm.\ 1.3]{cc2} that $c_k(Y,\lambda)$ does not depend on $J$. The numbers $c_k(Y,\lambda)$ are called the {\em ECH spectrum\/} of $(Y,\lambda)$.

If $\lambda$ is a degenerate contact form on $Y\approx S^3$ giving the tight contact structure, we define
\begin{equation}
\label{eqn:spectrumlimit}
c_k(Y,\lambda) = \lim_{n\to \infty}c_k(Y,f_n\lambda)
\end{equation}
where $\{f_n\}_{n\ge1}$ is a sequence of positive functions on $Y$ which converges to $1$ in the $C^0$ topology such that each contact form $f_n\lambda$ is nondegenerate. Lemmas from \cite[\S3.1]{qech} imply that this is well-defined, as explained in \cite[\S2.5]{two}.

Now let $X\subset\R^4$ be a compact star-shaped domain with smooth boundary $Y$. Then
\[
\lambda_{std} = \frac{1}{2}\sum_{i=1}^2(x_idy_i-y_idx_i)
\]
restricts to a contact form on $Y$, and we define the {\em ECH capacities\/} of $X$ by 
\begin{equation}
\label{eqn:definecapacity}
c_k(X) = c_k(Y,{\lambda_{std}}|_Y).
\end{equation}

\subsection{The combinatorial chain complex}
\label{sec:comb}

Let $X_\Omega$ be a concave toric domain determined by a convex function $f:[0,a]\to[0,b]$. We assume below that the function $f$ is smooth, $f'(0)$ and $f'(a)$ are irrational, $f'$ is constant near $0$ and $a$, and $f''(x)>0$ whenever $f'(x)$ is rational. Then $\partial X_\Omega$ is smooth. As we will see in \S\ref{sec:approx} below, $\lambda_{std}$ restricts to a degenerate contact form on $\partial X_\Omega$. Similarly to \cite{t3}, there is a combinatorial model for the ECH chain complex of appropriate nondegenerate perturbations of this contact form, which we denote by $ECC^{comb}_*(\Omega)$ and define as follows.

\begin{definition}
\label{def:combgen}
A generator of $ECC_*^{comb}(\Omega)$ is a quadruple $\widetilde{\Lambda}=(\Lambda,\rho,m,n)$, where:
\begin{description}
\item{(a)}
 $\Lambda$ is a concave integral path from $[0,B]$ to $[A,0]$ such that the slope of each edge is in the interval $[f'(0),f'(a)]$.
\item{(b)}
$\rho$ is a labeling of each edge of $\Lambda$ by `$e$' or `$h$'.
\item{(c)}
 $m$ and $n$ are nonnegative integers.
\end{description}
Here an ``edge'' of $\Lambda$ means a segment of $\Lambda$ of which each endpoint is either an initial or a final endpoint of $\Lambda$, or a point at which $\Lambda$ changes slope.
\end{definition}

We define the grading $I^{comb}(\widetilde{\Lambda})\in\Z$ of the generator $\widetilde{\Lambda}=(\Lambda,\rho,m,n)$ as follows. Let $\Lambda_{m,n}$ denote the path in the plane obtained by concatenating the following three paths:
\begin{description}
\item{(1)}
The highest polygonal path with vertices at lattice points from $(0,B+n+\floor{-mf'(0)})$ to $(m,B+n)$ which is below the line through $(m,B+n)$ with slope $f'(0)$.
\item{(2)}
The image of $\Lambda$ under the translation $(x,y)\mapsto (x+m,y+n)$. \item{(3)}
The highest polygonal path with vertices at lattice points from $(A+m,n)$ to $(A+m+\floor{-n/f'(a)},0)$ which is below the line through $(A+m,n)$ with slope $f'(a)$.
\end{description}
Let $\mc{L}(\Lambda_{m,n})$ denote the number of lattice points in the region bounded by $\Lambda_{m,n}$ and the axes, not including lattice points on the image of $\Lambda$ under the translation $(x,y)\mapsto (x+m,y+n)$. We then define
\begin{equation}
\label{eqn:icomb}
I^{comb}(\widetilde{\Lambda}) = 2\mc{L}(\Lambda_{m,n}) + h(\widetilde{\Lambda})
\end{equation}
where $h(\widetilde{\Lambda})$ denotes the number of edges of $\Lambda$ that are labeled `$h$'.

We define the action $\mc{A}^{comb}(\widetilde{\Lambda})\in\R$ of the generator $\widetilde{\Lambda}=(\Lambda,\rho,m,n)$ by
\begin{equation}
\label{eqn:acomb}
\mc{A}^{comb}(\widetilde{\Lambda}) = \ell_\Omega(\Lambda) + an+bm.
\end{equation}

One can also define a combinatorial differential on the chain complex $ECC_*^{comb}(\Omega)$ similarly to \cite{t3}, which agrees with the ECH differential for appropriate perturbations of the contact form and almost complex structures, but we do not need this here. What we do need is the following:

\begin{lemma}
\label{lem:approx}
For each $\epsilon>0$, there exists a contact form $\lambda$ on $\partial X_\Omega$ with the following properties:
\begin{description}
\item{(a)}
$\lambda$ is nondegenerate.
\item{(b)}
$\lambda=f\lambda_{std}|_{\partial X_\Omega}$ where $\|f-1\|_{C^0}<\epsilon$.
\item{(c)}
There is a bijection between the generators of $ECC(\partial X_\Omega,\lambda)$ with $\mc{A}<1/\epsilon$ and the generators of $ECC^{comb}(\Omega)$ with $\mc{A}^{comb}<1/\epsilon$, such that if $\alpha$ and $\widetilde{\Lambda}$ correspond under this bijection, then
\[
I(\alpha) = I^{comb}(\widetilde{\Lambda})
\]
and
\[
|\mc{A}(\alpha) - \mc{A}^{comb}(\widetilde{\Lambda})| < \epsilon.
\]
\end{description}
\end{lemma}

Lemma~\ref{lem:approx} will be proved in \S\ref{sec:approx}. We can now deduce:

\begin{lemma}
\label{lem:comb}
For each nonnegative integer $k$, there exists a generator $\widetilde{\Lambda}$ of $ECC^{comb}(\Omega)$ such that $I^{comb}(\widetilde{\Lambda})=2k$ and $\mc{A}^{comb}(\widetilde{\Lambda}) = c_k(X_\Omega)$.
\end{lemma}

\begin{proof}
Fix $k$. For each positive integer $n$, let $\lambda_n$ be a contact form provided by Lemma~\ref{lem:approx} for $\epsilon=1/n$. It follows from \eqref{eqn:spectrumlimit} and \eqref{eqn:definecapacity} that we can choose $\lambda_n$ so that
\[
|c_k(X_\Omega) - c_k(\partial X_\Omega,\lambda_n)|<1/n.
\]
By definition, there exists a generator $\alpha_n$ of $ECC_{2k}(\partial X_\Omega,\lambda_n)$ with $\mc{A}(\alpha_n)=c_k(\partial X_\Omega,\lambda_n)$.
Assume $n$ is sufficiently large that $c_k(X_\Omega) + 1/n < n$.  Then $\mc{A}(\alpha_n)<n$, so $\alpha_n$ corresponds to a generator $\widetilde{\Lambda}_n$ of $ECC^{comb}(\Omega)$ under the bijection in Lemma~\ref{lem:approx}, with
\begin{equation}
\label{eqn:comb1}
I^{comb}(\widetilde{\Lambda}_n)=2k
\end{equation}
and
\begin{equation}
\label{eqn:comb2}
|\mc{A}^{comb}(\widetilde{\Lambda}_n) - c_k(X_\Omega)| < 2/n.
\end{equation}
It follows from \eqref{eqn:icomb} that there are only finitely many generators $\widetilde{\Lambda}$ of $ECC^{comb}(\Omega)$ with $I^{comb}(\widetilde{\Lambda})=2k$. Consequently, there exists such a generator $\widetilde{\Lambda}$ which agrees with infinitely many $\widetilde{\Lambda}_n$. It now follows from \eqref{eqn:comb1} and \eqref{eqn:comb2} that $I^{comb}(\widetilde{\Lambda})=2k$ and $\mc{A}^{comb}(\widetilde{\Lambda}_n)=c_k(X_\Omega)$ as desired.
\end{proof}

\begin{proof}[Proof of Lemma~\ref{lem:pathub}.]
Fix $k$. By the continuity in Lemmas~\ref{lem:continuity1} and \ref{lem:continuity2}, we can assume that $\Omega$ is determined by a function $f:[0,a]\to[0,b]$ satisfying the conditions at the beginning of \S\ref{sec:comb}, such that in addition
\begin{equation}
\label{eqn:turn}
|f'(0)|, |1/f'(a)| > k.
\end{equation}
By Lemma~\ref{lem:comb}, we can choose a generator $\widetilde{\Lambda}=(\Lambda,\rho,m,n)$ of $ECC^{comb}(\Omega)$ with $I^{comb}(\widetilde{\Lambda})=2k$ and $\mc{A}^{comb}(\widetilde{\Lambda}) = c_k(X_\Omega)$. It follows from \eqref{eqn:turn} that $m=n=0$; otherwise the region bounded by $\Lambda_{m,n}$ and the axes would include at least $k+1$ lattice points on the axes not in the translate of $\Lambda$, so by \eqref{eqn:icomb} we would have $I^{comb}(\widetilde{\Lambda})>2k$, which is a contradiction.

Let $k'=\mc{L}(\Lambda)$. Then by \eqref{eqn:icomb} we have $k'\le k$, and by \eqref{eqn:acomb} we have $\ell_\Omega(\Lambda) = c_k(X_\Omega)$. Thus
\[
c_k(X_\Omega) \le \max\{\ell_\Omega(\Lambda) \mid \mc{L}(\Lambda) = k'\}.
\]
To complete the proof of Lemma~\ref{lem:pathub}, one could give a combinatorial proof that the right hand side of \eqref{eqn:pathub} is a nondecreasing function of $k$. Instead we will take a shortcut: by Lemma~\ref{lem:pathlb} we have
\[
\max\{\ell_\Omega(\Lambda) \mid \mc{L}(\Lambda) = k'\} \le c_{k'}(X_\Omega),
\]
and by \eqref{eqn:capacitiesincreasing} we have
\[
c_{k'}(X_\Omega) \le c_k(\Omega).
\]
Thus the above three inequalities are equalities.
\end{proof}

\subsection{The generators of the ECH chain complex}
\label{sec:approx}

To complete the computations of ECH capacities, our remaining task is to give the:

\begin{proof}[Proof of Lemma~\ref{lem:approx}.] The proof has five steps.

{\em Step 1.\/}
We first determine the embedded Reeb orbits of the contact form $\lambda_{std}|_{\partial X_\Omega}$ and their symplectic actions. Similarly to \cite[\S4.3]{bn}, these are given as follows:
\begin{itemize}
\item The circle $\gamma_1=\{z\in\partial X_{\Omega}\mid z_2=0\}$ is an embedded elliptic Reeb orbit with action $\mc{A}(\gamma_1)=a$.
\item The circle $\gamma_2=\{z\in\partial X_\Omega\mid z_1=0\}$ is an embedded elliptic Reeb orbit with action $\mc{A}(\gamma_2)=b$.
\item For each $x\in(0,a)$ such that $f'(x)$ is rational, the torus
\[
\{z\in\partial X_\Omega\mid \pi(|z_1|^2,|z_2|^2)=(x,f(x))\}
\]
is foliated by a Morse-Bott circle of Reeb orbits.
Let $v_1$ be the smallest positive integer such that $v_2=f'(x)v_1\in\Z$, write $v=(v_1,v_2)$, and denote this circle of Reeb orbits by $\mc{O}_v$. Then each Reeb orbit in $\mc{O}_v$ has symplectic action
\[
\mc{A} = v \times (x,f(x)).
\]
\end{itemize}

In particular, if $\alpha=\{(\alpha_i,m_i)\}$ is a finite set of embedded Reeb orbits with positive integer multiplicities, then $\alpha$ determines a triple $(\Lambda,m,n)$ satisying conditions (a) and (c) in Definition~\ref{def:combgen}. The path $\Lambda$ is obtained by taking the vector $v$ for each Reeb orbit $\alpha_i$ that is in the Morse-Bott circle $\mc{O}_v$, multiplied by the covering multiplicity $m_i$, and concatenating these vectors in order of increasing slope. The integer $m$ is the multiplicity of $\gamma_2$ if it appears in $\alpha$, and otherwise $m=0$; likewise $n$ is the multiplicity of $\gamma_1$ if it appears in $\alpha$ and otherwise $n=0$. It follows from the above calculations that
\[
\mc{A}(\alpha) = \ell_\Omega(\Lambda)+an+bm.
\]

{\em Step 2.\/}
Given $\epsilon>0$, we can now perturb $\lambda_{std}|_{\partial X_\Omega}$ to $\lambda=f\lambda_{std}|_{\partial X_\Omega}$ where $f$ is $C^0$-close to $1$, so that each Morse-Bott circle $\mc{O}_v$ of embedded Reeb orbits with action less than $1/\epsilon$ becomes two embedded Reeb orbits of approximately the same action, namely an elliptic orbit $e_v$ and a hyperbolic orbit $h_v$; no other Reeb orbits of action less than $1/\epsilon$ are created; and the Reeb orbits $\gamma_1$ and $\gamma_2$ are unaffected.

Now the generators of $ECC(\partial X_\Omega,\lambda)$ with $\mc{A}<1/\epsilon$ correspond to generators of $ECC^{comb}(\Omega)$ with $\mc{A}^{comb}<1/\epsilon$. Given a generator $\alpha=\{(\alpha_i,m_i)\}$ of $ECC(\partial X_\Omega,\lambda)$ with $\mc{A}(\alpha)<1/\epsilon$, the corresponding combinatorial generator $\widetilde{\Lambda}=(\Lambda,\rho,m,n)$ is determined as follows. The triple $(\Lambda,m,n)$ is determined as in Step 1. The labeling $\rho$ is defined as follows. Suppose an edge of $\Lambda$ corresponds to the vector $kv$ where $v=(v_1,v_2)$ is an irreducible integer vector and $k$ is a positive integer. Then either $\alpha$ contains the elliptic orbit $e_v$ with multiplicity $k$, or $\alpha$ contains the elliptic orbit $e_v$ with multiplicity $k-1$ and the hyperbolic orbit $h_v$ with multiplicity $1$. The labeling of the edge is `$e$' in the former case and `$h$' in the latter case.

To complete the proof of Lemma~\ref{lem:approx}, we need to show that $I(\alpha)=I^{comb}(\widetilde{\Lambda})$.

{\em Step 3.\/} Let $\alpha=\{(\alpha_i,m_i)\}$ be a generator of $ECC(\partial X_\Omega,\lambda)$. We now review the definition of the grading $I(\alpha)$ in the present context; for details of the grading in general see \cite[\S3]{bn} or \cite[\S2]{ir}. The formula is
\begin{equation}
\label{eqn:Ialpha}
I(\alpha) = c_\tau(\alpha) + Q_\tau(\alpha) + CZ_\tau^I(\alpha)
\end{equation}
where the individual terms are defined as follows. First, $\tau$ is a homotopy class of symplectic trivialization of $\xi=\Ker(\lambda)$ over each of the Reeb orbits $\alpha_i$. Next, $c_\tau(\alpha)$ is the relative first Chern class, with respect to $\tau$, of $\xi$ restricted to a surface bounded by $\alpha$. That is, if $\Sigma$ is a compact oriented surface with boundary and $g:\Sigma\to\partial X_\Omega$ is a smooth map such that $g(\partial\Sigma)=\sum_im_i\alpha_i$, then $c_\tau(\alpha)$ is the algebraic count of zeroes of a section of $g^*\xi$ which on each boundary circle is nonvanishing and has winding number zero with respect to $\tau$. The relative first Chern class is additive in the sense that
\[
c_\tau(\alpha) = \sum_im_ic_\tau(\alpha_i).
\]
Next, $Q_\tau(\alpha)$ is the relative self-intersection number; in the present situation this is given by
\begin{equation}
\label{eqn:Qtau}
Q_\tau(\alpha) = \sum_im_i^2Q_\tau(\alpha_i) + \sum_{i\neq j}m_im_j\op{link}(\alpha_i,\alpha_j).
\end{equation}
Here $Q_\tau(\alpha_i)$ is the linking number of $\alpha_i$ with a pushoff of itself via the trivialization $\tau$, and $\op{link}(\alpha_i,\alpha_j)$ denotes the linking number of $\alpha_i$ and $\alpha_j$. Finally,
\[
CZ_\tau^I(\alpha)=\sum_i\sum_{k=1}^{m_i}CZ_\tau(\alpha_i^k)
\]
where $CZ_\tau(\alpha_i^k)$ denotes the Conley-Zehnder index of the $k$-fold iterate of $\alpha_i$ with respect to the trivialization $\tau$. In particular, if $\gamma$ is an elliptic orbit such that the linearized Reeb flow around $\gamma$ with respect to the trivialization $\tau$ is conjugate to a rotation by $2\pi\theta$ for $\theta\in\R/\Q$, then
\[
CZ_\tau(\gamma^k)=2\floor{k\theta}+1.
\]

{\em Step 4.\/} We now calculate the terms that enter into the grading formula \eqref{eqn:Ialpha} when $\alpha$ is a generator of $ECC(\partial X_\Omega,\lambda)$ with $\mc{A}(\alpha)<1/\epsilon$.

We first choose a trivialization $\tau$ of $\xi$ over each embedded Reeb orbit of action less than $1/\epsilon$. There is a distinguished trivialization $\tau$ of $\xi$ over $\gamma_1$ determined by the disk in the plane $z_2=0$ bounded by $\gamma_1$. With respect to this trivialization, the linearized Reeb flow around $\gamma$ is rotation by $-2\pi/f'(a)$, so that
\begin{equation}
\label{eqn:cz1}
CZ_\tau(\gamma_1^k)=2\floor{-k/f'(a)}+1.
\end{equation}
Likewise, there is a distinguished trivialization $\tau$ of $\xi$ over $\gamma_2$ determined by the disk in the plane $z_1=0$ bounded by $\gamma_2$. With respect to this trivialization, we have
\begin{equation}
\label{eqn:cz2}
CZ_\tau(\gamma_2^k)=2\floor{-kf'(0)}+1.
\end{equation}
We also have
\[
c_\tau(\gamma_i)=1, \quad\quad Q_\tau(\gamma_i)=0
\]
for $i=1,2$.

We can choose the trivialization $\tau$ over the orbits $e_v$ and $h_v$ coming from the Morse-Bott circles so that the linearized Reeb flow around $e_v$ is a slight negative rotation, and the linearized Reeb flow around $h_v$ does not rotate the eigenspaces of the linearized return map. This implies that
\begin{equation}
\label{eqn:czeh}
CZ_\tau(e_v^k)=-1, \quad\quad CZ_\tau(h_v)=0
\end{equation}
whenever $k$ is sufficiently small that $e_v^k$ has action less than $1/\epsilon$.
We also have
\begin{gather*}
c_\tau(e_v)=c_\tau(h_v)=v_1-v_2,\\
Q_\tau(e_v)=Q_\tau(h_v)=\op{link}(e_v,h_v)=-v_1v_2.
\end{gather*}
Finally, the linking numbers of pairs of distinct embedded Reeb orbits are given as follows. Below, $o_v$ denotes either $e_v$ or $h_v$.
\[
\begin{split}
\op{link}(\gamma_1,\gamma_2) &= 1,\\
\op{link}(\gamma_1,o_v) &= -v_2,\\
\op{link}(\gamma_2,o_v) &= v_1,\\
\op{link}(o_v,o_w) &= \min(-v_1w_2,-v_2w_1).
\end{split}
\]

{\em Step 5.\/}
Let $\alpha$ and $\widetilde{\Lambda}$ be as in Step 2; we compute the grading $I(\alpha)$ in terms of $\widetilde{\Lambda}=(\Lambda,\rho,m,n)$.

As in \S\ref{sec:comb}, let $(0,B)$ and $(A,0)$ denote the endpoints of $\Lambda$. The Chern class calculations in Step 4 then imply that
\begin{equation}
\label{eqn:ctaualpha}
c_\tau(\alpha) = A+B+m+n.
\end{equation}

Next, let $\Lambda_{m,n}'$ be defined like $\Lambda_{m,n}$ in \S\ref{sec:comb}, but with the first path replaced by a horizontal segement from $(0,B+n)$ to $(m,B+n)$, and with the third path replaced by a vertical segment from $(A+m,n)$ to $(A+m,0)$. Let $R_{m,n}'$ denote the region bounded by $\Lambda_{m,n}'$ and the axes. We then have
\[
Q_\tau(\alpha) = 2\op{Area}(R_{m,n}').
\]
This follows by expanding $Q_\tau(\alpha)$ using \eqref{eqn:Qtau} and the formulas in Step 4, and then interpreting the result as the area of $R_{m,n}'$ computed by appropriately dissecting it into right triangles and rectangles. Let $\mc{L}(\Lambda_{m,n}')$ denote the number of lattice points in $R_{m,n}'$, not including lattice points on the translate of $\Lambda$. Let $E$ denote the number of lattice points on $\Lambda$. By Pick's formula for the area of a lattice polygon, we have
\[
2\op{Area}(R_{m,n}') = 2\mc{L}(\Lambda_{m,n}') + E - 2m-2n-A-B-1.
\]
Let $e(\widetilde{\Lambda})$ denote the total multiplicity of all elliptic orbits in $\alpha$. Observe that
\[
E=e(\widetilde{\Lambda})+h(\widetilde{\Lambda})+1.
\]
Combining the above three equations, we obtain
\begin{equation}
\label{eqn:qalpha}
Q_\tau(\alpha) = 2\mc{L}(\Lambda_{m,n}') + e(\widetilde{\Lambda}) + h(\widetilde{\Lambda}) - 2m-2n-A-B.
\end{equation}

Finally, it follows from \eqref{eqn:cz1} and \eqref{eqn:cz2} that
\[
\sum_{k=1}^nCZ_\tau(\gamma_1^k) + \sum_{k=1}^mCZ_\tau(\gamma_2^k) = 2(\mc{L}(\Lambda_{m,n})-\mc{L}(\Lambda_{m,n}')) + m + n.
\]
By \eqref{eqn:czeh}, the sum of the remaining Conley-Zehnder terms in $CZ_\tau^I(\alpha)$ is $-e(\widetilde{\Lambda})$. Thus
\begin{equation}
\label{eqn:czalpha}
CZ_\tau^I(\alpha) = 2(\mc{L}(\Lambda_{m,n})-\mc{L}(\Lambda_{m,n}')) + m + n -e(\widetilde{\Lambda}).
\end{equation}

Adding equations \eqref{eqn:ctaualpha}, \eqref{eqn:qalpha}, and \eqref{eqn:czalpha} gives 
\[
I(\alpha) = 2\mc{L}(\Lambda_{m,n}) + h(\widetilde{\Lambda})
\]
as desired.
\end{proof}

\section{The union of an ellipsoid and a cylinder}
\label{sec:uec}

In this section we study symplectic embeddings into $Z(a,b,c)$, which is the union of the cylinder $Z(a)$ with the ellipsoid $E(b,c)$. In \S\ref{sec:oee} we give a generalization of Corollary~\ref{cor:ballcylinder}, and in \S\ref{sec:cbp} and \S\ref{sec:obp} we prove Proposition~\ref{prop:packing} and Theorem~\ref{thm:packing}.

\subsection{Optimal ellipsoid embeddings}
\label{sec:oee}

We now prove the following proposition which asserts that certain inclusions of an ellipsoid into the union of an ellipsoid and a cylinder are optimal. This is a generalization of Corollary~\ref{cor:ballcylinder}, which is the case $b=c$. 
\begin{proposition}
\label{prop:ellipsoidcylinder}
Let $a$ be a positive integer and let $b,c$ and $\lambda$ be positive real numbers. Assume $c>1$, $a\ge b/c$, and at least one of the following two conditions:
\begin{description}
\item{(i)} $a = \floor{b/c}+1$.
\item{(ii)}
$b \le \frac{c}{c-1}$.
\end{description}
Then there exists a symplectic embedding $E(a,1)\to Z(\lambda,\lambda b,\lambda c) $ if and only if $E(a,1) \subset Z(\lambda,\lambda b,\lambda c)$.
\end{proposition}

\begin{proof}
Using $c>1$ and $a\ge b/c$, we calculate that $E(a,1) \subset Z(\lambda,\lambda b,\lambda c))$ if and only if
\begin{equation}
\label{eqn:ellipsoidinclusion}
\lambda\ge \frac{ac}{ac+b(c-1)}.
\end{equation}
Consequently, as in the proof of Corollary~\ref{cor:ballcylinder}, Proposition~\ref{prop:ellipsoidcylinder} follows from the Ellipsoid axiom and Lemma~\ref{lem:capacitycomputation} below.
\end{proof}

\begin{lemma}
\label{lem:capacitycomputation}
Let $a$ be a positive integer and let $b$ and $c$ be positive real numbers with $c > 1$ and $a \ge b/c$.  Assume that (i) or (ii) in Proposition~\ref{prop:ellipsoidcylinder} holds.  Then 
\begin{equation}
\label{eqn:zbc}
c_a(Z(1,b,c)) = a + \frac{b(c-1)}{c}.
\end{equation}
\end{lemma}
\begin{proof}
The proof has three steps.

{\em Step 1.\/} We first prove equation \eqref{eqn:zbc} in case (ii) when $c\ge b$.

Referring back to the definition of the weight expansion in \S\ref{sec:weight}, we have
\[
\begin{split}
X_{\Omega_1} &= B\left(\frac{b(c-1)}{c}+1\right),\\
X_{\Omega_2} &= E\left(\frac{b(c-1)}{c},\frac{(c-b)(c-1)}{c}\right),\\
X_{\Omega_3} &= Z(1).
\end{split}
\]
By Theorem~\ref{thm:weight} and the Disjoint Union property of ECH capacities, we have
\[
c_a(Z(1,b,c)) = \max_{k_1+k_2+k_3=a} \sum_{i=1}^3c_{k_i}(X_{\Omega_i}).
\]
Now $c_{k_3}(X_{\Omega_3})=k_3$. Also, it follows from the Ellipsoid property that $c_k(E(\alpha,\beta))\le k\alpha$. Since we are assuming that $\frac{b(c-1)}{c}\le 1$, we deduce that $c_{k_2}(X_{\Omega_2})\le k_2$. Thus the maximum is achieved with $k_2=0$. Since $1<\frac{b(c-1)}{c}+1\le 2$, it follows as in \eqref{eqn:cab} that the maximum is achieved with $k_1=1$. Equation \eqref{eqn:zbc} follows.

{\em Step 2.\/} We now prove equation \eqref{eqn:zbc} in case (ii) when $b\ge c$.

Here, in the inductive definition of the weight expansion, the first $\floor{b/c}$ steps yield $\floor{b/c}$ copies of the ball $B(c)$. The remaining region is $Z(1, b-c\floor{b/c},c)$. Here, if $c$ divides $b$, then we regard $Z(1, b-c\floor{b/c},c)$ as $Z(1)$. Thus by Theorem~\ref{thm:weight} and the Disjoint Union property,
\begin{equation}
\label{eqn:max}
c_a(Z(1,b,c)) = \max_{k_1+k_2=a}\left(c_{k_1}\left(E\left(c,c\floor{b/c}\right)\right) + c_{k_2}\left(Z(1, b-c\floor{b/c},c\right))\right).
\end{equation}
Step 1 applies to show that
\begin{equation}
\label{eqn:ck2}
c_{k_2}\left(Z(1, b-c\floor{b/c},c\right)) = k_2 + \frac{(b-c\floor{b/c})(c-1)}{c}
\end{equation}
whenever $k_2\ge 1$. Also, it follows from the Ellipsoid property that
\begin{equation}
\label{eqn:ck1}
c_{k_1}(E(c,c\floor{b/c})) = ck_1
\end{equation}
for $k_1\le \floor{b/c}$. For larger values of $k_1$, one has to increase $k_1$ by at least $\floor{b/c}\ge 2$ to obtain any increase in $c_{k_1}(E(c,c\floor{b/c}))$, and this increase will always equal $c$. On the other hand, it follows from $b\ge c$ and (ii) that $c\le 2$. Hence the maximum is attained for $k_1\le \floor{b/c}$. Since $a\ge\floor{b/c}$ and $c>1$, the maximum is attained with $k_1=\floor{b/c}$. Adding \eqref{eqn:ck2} and \eqref{eqn:ck1} then proves \eqref{eqn:zbc}.

{\em Step 3.\/} We now prove equation \eqref{eqn:zbc} in case (i). As in Step 2, the first $a$ steps of the weight expansion yield $a-1$ copies of the ball $B(c)$, together with the ball $B\left(\frac{(b-c(a-1))(c-1)}{c}+1\right)$. It follows from Lemma~\ref{lem:approximation} that 
\[
c_a(Z(1,b,c)) = (a-1)c + \frac{(b-c(a-1))(c-1)}{c}+1.
\]
Simplifying this expression gives equation \eqref{eqn:zbc} again.
\end{proof}

\begin{remark}
\label{rmk:alebc}
If $c > 1$ and $a$ is a positive integer with $a\le b/c$, then
\begin{equation}
\label{eqn:alebc}
c_a(Z(1,b,c))=ac.
\end{equation}
This is because the first $a$ steps in the weight expansion yield $a$ copies of the ball $B(c)$, and we can then apply Lemma~\ref{lem:approximation}.
\end{remark}

\subsection{Construction of ball packings}
\label{sec:cbp}

\begin{proof}[Proof of Proposition~\ref{prop:packing}] 
The proof has three steps.

{\em Step 1.}  Choose $k\in\{1,\ldots,n\}$ maximizing $\lambda_k$.  We claim that
$\lambda_{k}\geq w_{i}$
for all $i>k$.

To see this, use \eqref{eqn:lambdak} to compute that
\[
\lambda_{k}-w_{k+1}  =\left(k+\frac{b(c-1)}{c}+1\right)\left(\lambda_{k}-\lambda_{k+1}\right).
\]
Since $\lambda_k$ is maximal, we deduce that $\lambda_k\ge w_{k+1}$. The rest follows from the fact that $w_1\ge\cdots\ge w_n$.

{\em Step 2.\/} Let $\Omega$ be the region for which $X_\Omega = Z(\lambda,\lambda b,\lambda c)$. That is, $\Omega$ is bounded by the axes, the line segment from $(0,\lambda c)$ to $\left(\frac{b}{c}(c-1)\lambda,\lambda\right)$, and the horizontal ray extending to the right from the latter point. By Lemma~\ref{lem:traynortrick}, it suffices to embed disjoint open triangles $T_1,\ldots,T_n$ into $\Omega$, such that $T_\ell$ is affine equivalent to $\bigtriangleup(w_\ell)$ for each $\ell$. If $\ell>k$, then by Step 1 we have $\lambda\ge w_\ell$, so we can simply take $T_\ell$ to be a translate of $\bigtriangleup(w_\ell)$ sufficiently far to the right.

{\em Step 3.\/} For $1\le \ell\le k$, we now define the triangle $T_\ell$ by starting with the triangle $\bigtriangleup(w_\ell)$, multiplying by $\begin{pmatrix} 1 & -(\ell-1) \\ 0 & 1 \end{pmatrix}\in SL_2\Z$, and then translating to the right by $\sum_{i=1}^{\ell-1}w_i$. The vertices of $T_\ell$ are
\[
\left(\sum_{i=1}^{\ell-1}w_{i},0\right),\;\left(\sum_{i=1}^{\ell}w_{i},0\right),\;\textrm{and}\;\left(\sum_{i=1}^{\ell-1}w_{i}-(\ell-1)w_{\ell},w_{\ell}\right).
\]
Observe that the right edge of $T_\ell$ has slope $-1/\ell$; and if $\ell>1$ then the left edge of $T_\ell$ is a subset of the right edge of $T_{\ell-1}$. In particular, the triangles $T_1,\ldots,T_k$ are disjoint; and the upper boundary of the union of their closures, call this path $\Lambda$, is the graph of a convex function.

To verify that the triangles $T_1,\ldots,T_k$ are contained in $\Omega$, we need to check that the path $\Lambda$ does not go above the upper boundary of $\Omega$, see Figure~\ref{fig: the embedding into Z}. The initial endpoint of $\Lambda$ is $(0,w_1)$, which is not above the upper boundary of $\Omega$ by our assumption that $\lambda \ge w_1/c$. Next, $\Lambda$ crosses the horizontal line of height $\lambda$ at the point $\left(\sum_{\ell=1}^k (w_\ell-\lambda),\lambda\right)$. By convexity, it is enough to check that this point is not to the right of the corner $\left(\frac{b}{c}(c-1)\lambda,\lambda\right)$ of $\partial\Omega$.
This holds because
\[
\lambda\ge\lambda_{k}=\frac{\sum_{\ell=1}^{k}w_{\ell}}{k+\frac{b}{c}(c-1)}
\]
implies that
\[
\sum_{\ell=1}^{k}\left(w_{\ell}-\lambda\right)\le\frac{b}{c}(c-1)\lambda.
\]

\begin{figure}
\begin{center}
\scalebox{0.35}{\includegraphics{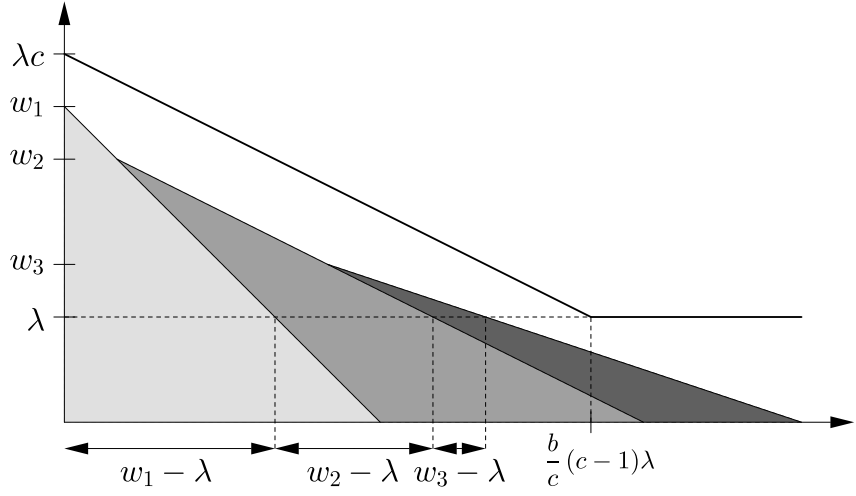}}
\end{center}
\caption{The union of the triangles $T_{\ell}$ is a subset of $\Omega$}
\label{fig: the embedding into Z}
\end{figure}

\end{proof}

\subsection{The ECH obstruction to ball packings}
\label{sec:obp}

We now complete the proof of Theorem~\ref{thm:packing}. By Proposition~\ref{prop:packing}, it is enough to prove:

\begin{lemma}
Under the assumptions of Theorem~\ref{thm:packing}, if there exists a symplectic embedding
\[
\coprod_{i=1}^{n} \op{int}(B\left(w_{i}\right)) \to Z(\lambda,\lambda b,\lambda c),
\]
then $\lambda\ge \max\{w_1/c,\lambda_1,\ldots,\lambda_n\}$.
\end{lemma}

\begin{proof}
By the Monotonicity and Conformality properties of ECH capacities, it is enough to show that there is a positive integer $k$ such that 
\begin{equation}
\label{eqn:obstruction}
c_k\left(\coprod_{i=1}^{n} \op{int}(B\left(w_{i}\right))\right) \ge \max\{w_1/c,\lambda_1,\ldots,\lambda_n\}\cdot c_k(Z(1,b,c)).
\end{equation}  
By the Disjoint Union axiom, if $1\le k\le n$ then
\[
c_k\left(\coprod_{i=1}^{n} \op{int}(B\left(w_{i}\right))\right) \ge \sum_{i=1}^{k} w_i.
\]
So to prove \eqref{eqn:obstruction}, it is enough to show that there exists $k\in\{1,\ldots,n\}$ with
\begin{equation}
\label{eqn:obstruction2}
\sum_{i=1}^kw_i \ge \max\{w_1/c,\lambda_1,\ldots,\lambda_n\}\cdot c_k(Z(1,b,c)).
\end{equation}
We will prove this by considering two cases.

{\em Case 1.}  Assume that $b \le c$.  Then $w_1/c \le \lambda_1$.  Hence
\begin{equation}
\label{eqn:maxlambdak}
\max\left\{ w_{1}/c,\lambda_{1},\ldots,\lambda_{n}\right\}=\max\{\lambda_1,\ldots,\lambda_n\}.
\end{equation}
We claim now that \eqref{eqn:obstruction2} holds for $k\in\{1,\ldots,n\}$ maximizing $\lambda_k$. To prove this, we need to show that
\[
\sum_{i=1}^kw_i \ge \lambda_k c_k(Z(1,b,c)).
\]
By equation \eqref{eqn:lambdak}, the above inequality is equivalent to
\begin{equation}
\label{eqn:obstruction3}
c_k(Z(1,b,c)) \le k + \frac{b(c-1)}{c}.
\end{equation}
Since
 $b/c \le 1$, it follows from Lemma~\ref{lem:capacitycomputation} that equality holds in \eqref{eqn:obstruction3}.

{\em Case 2.}  Assume that $b \ge c$.  By Corollary~\ref{cor:gromovwidth}, we have
\[
c_1(Z(1,b,c))=c.
\]
Consequently, we can assume without loss of generality that \eqref{eqn:maxlambdak} holds, since otherwise the inequality \eqref{eqn:obstruction2} holds for $k=1$. 
As in Step 1, it is now enough to prove the inequality \eqref{eqn:obstruction3}, where $k\in\{1,\ldots,n\}$ maximizes $\lambda_k$.

If $k \ge b/c$, then equality holds in \eqref{eqn:obstruction3} by Lemma~\ref{lem:capacitycomputation}.  If $k < b/c$, then the inequality \eqref{eqn:obstruction3} follows from Remark~\ref{rmk:alebc}, since in this case
\[
kc < k + \frac{b(c-1)}{c}.
\]
\end{proof}

\end{document}